\documentclass[12pt,a4paper]{amsart}
\usepackage{graphics}
\usepackage{epsfig}
\usepackage{graphicx}
\theoremstyle{plain}
\usepackage{amssymb}
\usepackage[ukrainian,english]{babel}
\advance\hoffset-20mm \advance\textwidth40mm

\newtheorem{theorem}{Theorem}
\newtheorem{lemma}{Lemma}
\newtheorem*{theo*}{Theorem}

\newtheorem{corollary}{Corollary}
\theoremstyle{definition}

\newtheorem*{definition*}{Definition}
\newtheorem{example}{Example}
\newtheorem{remark}{Remark}

\DeclareMathOperator{\Ker}{Ker}

\DeclareMathOperator{\diver}{div}

\begin{document}
\sloppy
\title[Centralizers of Jacobian derivations]
{Centralizers of Jacobian derivations}
\author
{D.I.Efimov, A.P.Petravchuk, M.S.Sydorov}
\address{D.I.Efimov: Department of Algebra and Computer Mathematics, Faculty of Mechanics and Mathematics,
	Taras Shevchenko National University of Kyiv, 64, Volodymyrska street, 01033  Kyiv, Ukraine
}
\email{danil.efimov@yahoo.com}
\address{A.P.Petravchuk: Department of Algebra and Computer Mathematics, Faculty of Mechanics and Mathematics,
	Taras Shevchenko National University of Kyiv, 64, Volodymyrska street, 01033  Kyiv, Ukraine
}
\email{petravchuk@knu.ua , apetrav@gmail.com}
\address{M.S.Sydorov:
	Department of Algebra and Computer Mathematics, Faculty of Mechanics and Mathematics,
	Taras Shevchenko National University of Kyiv, 64, Volodymyrska street, 01033  Kyiv, Ukraine}
\email{smsidorov95@gmail.com}
\date{\today}

%
\begin{abstract}
Let $\mathbb K$ be an algebraically closed field of characteristic zero,
$\mathbb K[x, y]$ the polynonial ring in variables $x$, $y$ and let
$W_2(\mathbb K)$ be the Lie algebra of all $\mathbb K$-derivations on
$\mathbb K[x, y]$.
A derivation $D \in W_2(\mathbb K)$ is called a Jacobian derivation
if there exists $f \in \mathbb K[x, y]$ such that $D(h) = \det J(f, h)$
for any $h \in \mathbb K[x, y]$
(here $J(f, h)$ is the Jacobian matrix for $f$ and $h$).
Such a derivation is denoted by $D_f$.
The kernel of $D_f$ in $\mathbb K[x, y]$ is a subalgebra $\mathbb K[p]$
where $p=p(x, y)$ is a polynomial of smallest degree such that
$f(x, y) = \varphi (p(x, y)$ for some $\varphi (t) \in \mathbb K[t]$. Let $C = C_{W_2(\mathbb K)} (D_f)$ be the centralizer of $D_f$ in
$W_2(\mathbb K)$.
We prove that $C$ is the free  $\mathbb K[p]$-module of rank 1
or 2 over $\mathbb K[p]$ and point out a criterion  of being a module of rank $2$.
These results  are used to obtain a class of integrable autonomous systems  of differential equations.

 \end{abstract}
\maketitle

\section{Introduction}
Let $\mathbb K$ be an algebraically closed field of characteristic zero,
$\mathbb K[x, y]$ the polynomial ring in variables $x$, $y$ and
$R = \mathbb K(x, y)$ the field of rational functions.
Recall that a $\mathbb K$-linear map
$D: \mathbb K[x, y]  \mathbb \longrightarrow \mathbb K[x, y]$ is called
a $\mathbb K$-derivation (or a derivation if $\mathbb K$ is fixed) if
$D(fg) = D(f)g + fD(g)$ for any $f, g \in \mathbb K[x, y]$.
All the $\mathbb K$-derivations on $\mathbb K[x, y]$ form a Lie algebra
over $\mathbb K$ (denoted by $W_2(\mathbb K)$) with respect to the
operation $[D_1, D_2] = D_1D_2 - D_2D_1$.
Every element $D \in W_2(\mathbb K)$ can be uniquely written in the
form $D = f(x, y)\partial_x + g(x, y) \partial_y$, where
$\partial_x:=\frac{\partial}{\partial x}, \partial_y:=\frac{\partial}{\partial y}$ are partial derivatives on $\mathbb K[x, y]$. The latter means that $W_2(\mathbb K)$ is a free module of rank $2$ over $\mathbb K[x, y]$ and $\{ \partial _x, \partial _y\}$ is a free basis of this module.
The Lie algebra $W_2(\mathbb K)$ is, from the geometrical point of view,  the Lie algebra of all
polynomial vector fields on $\mathbb K^2$ and was studied intensively
from many points of view (see, for example, \cite{Lie}, \cite{Olver}, \cite{Petien2}).

Let $f\in \mathbb K[x, y]$.  The  polynomial $f$ defines a    derivation $D_f \in W_2(\mathbb K)$ by the rule: $D_f(h) = \det J(f, h)$ for any $h \in \mathbb K[x, y]$ (here $J(f, h)$ is the Jacobian matrix for $f$ and $h$). 
The derivation $D_f$ is called the Jacobian derivation
associated with the  polynomial $f$.
The kernel
$\Ker D_f$  in $\mathbb K[x, y]$ is an integrally closed subalgebra
of $\mathbb K[x, y]$ and $f \in \Ker D_f$. By \cite{Nag_Nov},
$\Ker D_f = \mathbb K [p]$, where $p$ is a generative closed  polynomial for $f$.

We study the structure of the centralizer $C_{W_2(\mathbb K)} (D_f)$.
This centralizer is of interest because from viewpoint of theory of ODE with any
derivation $D = f(x, y) \partial_x + g(x, y) \partial_y$ one can  associate 
an autonomous system of  ordinary differential equations 
$$ \frac{dx}{dt}=f(x, y), \ \frac{dy}{dt}=g(x, y)$$
and elements from $C_{W_2(\mathbb K)} (D)$ give
information about solutions of this system.

We give a criterion for a Jacobian derivation $D_f$ to have  the centralizer of rank 2
over $\Ker D_f$ (Theorem 1).
We also prove that  $C_{W_2(\mathbb K)}(D_f)$ is a free module
over the subalgebra $\mathbb K [p]$ of rank 1 or 2 (Theorem 2).
We  point out an example of integrable system of differential equations associated with a Jacobian derivation of special type.

We use standard notations. If $T = P \partial_x + Q \partial_y$
then the divergence $\diver T$ is defined as for a vector field with components $P, Q$: $\diver T = P_x' + Q_y'$.
If $T = P \partial_x + Q \partial_y$ is divergence-free (i.e., 
$\diver T = 0$), then $T = D_f$ for a polynomial $f$ that is a
``potential'' for the vector field determined by $T$. A polynomial $f\in \mathbb K[x, y]$ is called a closed polynomial if the subalgebra $\mathbb K[f]$ is integrally closed in the polynomial algebra $\mathbb K[x, y] $.  For any polynomial $f\in \mathbb K[x, y]$ there exists a closed polynomial $p(x, y) $ such that $f=\varphi (p)$ for some polynomial $\varphi \in \mathbb K[x, y].$ This  polynomial $p(x, y)$ will be called a generative closed polynomial for $f(x, y).$ If $L$ is a subalgebra of the Lie algebra $W_2(\mathbb K)$  then $\dim _RRL$ will be called the rank of $L$ and denoted by ${\rm rk}_{\mathbb K[x, y]}L $  or simply by ${\rm rk} L. $

\section{A criterion for centralizers to have rank $2$}
Some properties of derivations on polynomial rings are collected in the next lemma. 
\begin{lemma}\label{lem1}
	{\rm (1)}
	Let $D_1, D_2 \in W_2(\mathbb K)$ and $f, g \in \mathbb K[x, y ]$.
	Then 
	$$[fD_1, gD_2] = fD_1(g)D_2 -gD_2(f)D_1 + fg[D_1, D_2].$$
	
	{\rm (2)}
	If $f \in \mathbb K[x, y]$ and $p$ is a   generative closed  polynomial
	for $f$, then $\Ker D_f = \mathbb K[p]$.
	
	{\rm (3)} If $T\in W_2(\mathbb K)$ and ${\rm div}T=0$, then  $T=D_g$ for some polynomial $g\in \mathbb K[x, y].$ 
\end{lemma}
\begin{proof}
	(1) Direct calculation. (2)  See, for example, \cite{Petien}. (3) See, for example,  \cite{Now}.
\end{proof}
\begin{lemma}\label{lem2}
	Let $T \in W_2(\mathbb K)$ and
	$T(f) = \lambda f$ for some
	polynomials $f, \lambda \in \mathbb K[x, y]$.
	Then $[T, D_f] = D_{\lambda f} - (\diver T) D_f$.
\end{lemma}
\begin{proof}
	Let us write down  the derivation $T$ in the form
	$T = P \partial_x + Q \partial_y$ for some polynomials
	$P, Q \in \mathbb K[x, y]$. Then the condition
	$T(f) = \lambda f$ can be written in the form 
	\begin{equation}\label{e1}    Pf'_x + Qf'_y = \lambda f 
	\end{equation}	
	Let us differentiate the equality (\ref{e1}) on $x$ and then on $y$.
	We obtain
	\begin{equation}\label{e2}
		P'_xf'_x + Pf''_{x^2} + Q'_xf'_y + Qf''_{yx} =
		\lambda'_x f + \lambda f'_x,
	\end{equation}
	\begin{equation}\label{e3}
		P'_yf'_x + Pf''_{xy} + Q'_yf'_y + Qf''_{y^2} =
		\lambda'_y f + \lambda f'_y.
	\end{equation}
	Further, write down the product  of derivations $T$ and $D_f$ in terms of their components: 
	\begin{equation}\label{commuting}  
		[T, D_f] = (P'_xf'_y - P'_yf'_x - Pf''_{yx} - Qf''_{y^2}) \partial_x + \\
		(Pf''_{x^2} + Qf''_{xy} + f'_yQ'_x - f'_xQ'_y) \partial_y.
	\end{equation}
	Let us denote for convenience 
	$\alpha = -P'_y f'_x - Pf''_{yx} - Qf''_{y^2}$ and
	$\beta = Pf''_{x^2} + Qf''_{xy} + f'_y Q'_x$. Then using (\ref{e3}) and (\ref{e2}) we  see that
	\begin{equation}\label{e5}
		\alpha = Q'f'_y-\lambda '_yf-\lambda f'_y,  \  \beta =\lambda '_xf+\lambda f'_x-P'_xf'_x. 
	\end{equation} 
	
	The equality (\ref{commuting}) can be rewritten  in the form
	$$[T, D_f] = (P'_x f'_y + \alpha) \partial_x + (\beta - f'_x Q'_y) \partial_y.$$
	Inserting in the last equality instead $\alpha$ and $\beta$ their expressions from 
	(\ref{e5}) we see that 
	$$[T, D_f] = (P'_x f'_y - \lambda'_y f - \lambda f'_y + Q'_y f'_y ) \partial_x +
	(\lambda'_x f + \lambda f'_x - P'_x f'_x - f'_x Q'_y) \partial_y.$$
	After rearranging  the summands in the right part of this equality we get
	$$ [T, D_f] = ((\diver T) f'_y - (\lambda f)'_y) \partial_x +
	((\lambda f)'_x - (\diver T) f'_x )\partial_y $$
	The latter means that
	$$ [T, D_f] = ( \diver T )( f'_y \partial_x - f'_x \partial_y) +
	D_{\lambda f} = D_{\lambda f} - (\diver T)D_f.$$
	The proof is complete.
	
\end{proof}

\begin{remark}\label{rem1}
	The direct calculation  shows that $D_{\lambda f} = \lambda D_f + f D_\lambda$.
	Therefore
	$$ [T, D_f] = \lambda D_f + fD_\lambda - (\diver T)D_f = fD_\lambda + (\diver T-\lambda)D_f.$$
\end{remark}
\begin{lemma}\label{lem3} Let $T \in W_2(\mathbb K), f \in \mathbb K[x, y]$ be such
	that $[T, D_f] = 0$. If $T(f) = c$ for some $c \in \mathbb K$, then $T = D_g$ for some
	polynomial $g \in \mathbb K[x, y]$.
\end{lemma}
\begin{proof}
	
	Let us write down the derivation $T$ in the form $T = P\partial_x + Q\partial_y$ for some
	$P, Q \in \mathbb K[x, y]$. Then $Pf'_x + Qf'_y = c$ by conditions of the lemma. Differentiating this equality
	first on $x$ and then on $y$, we obtain the next equalities
	\begin{equation}\label{e6}
		P'_xf'_x + Pf''_{x^2} + Q'_xf'_y + Qf''_{yx} = 0,
	\end{equation}
	\begin{equation}\label{e7}
		P'_yf'_x + Pf''_{xy} + Q'_yf'_y + Qf''_{y^2} = 0.
	\end{equation}
	We see from (\ref{e6}) that 
	$$ Pf''_{x^2}+Q'_xf'_y=-P'_xf\_x-Qf''_{xy}$$
	and anagously from (\ref{e7}) 
	$$ P'_yf'_{x}+P_xf''_{xy}=-Q'_yf'_y-Qf''_{y^2}.$$
	Therefore it follows from (\ref{commuting}) that
	$$ [T, D_f] = (-Q'_yf'_x - P'_xf'_x)\partial_y + (P'_xf'_y + Q'_yf'_y)\partial_x = $$
	$$ = f'_y(P'_x + Q'_y) \partial_x - f'_x(P'_x + Q'_y) \partial_y = \diver T \cdot (-D_f). $$
	Since  $[T, D_f]=0$ by conditions of the lemma,  we see from the equality $[T, D_f]=-(\diver T)D_f$, that   $\diver T=0$. 
	It follows from Lemma 1 that $T = D_g$ for some polynomial $g \in \mathbb K[x, y].$
	
\end{proof}

\begin{corollary}\label{cor1}
	(1)
	Let $T \in C_{W_2(\mathbb K)}(D_f).$ If $T(f) = 0,$ then $T = \varphi D_f$ for
	some rational  function $\varphi \in \mathbb K(x, y)$ such that $\varphi \in \Ker D_f.$
	(2)
	If $T(f) = c$ for some $c \in \mathbb K^*$, then $T = D_g$ for some
	$g \in \mathbb K[x, y]$ such that $f, g$ form a  Jacobian pair,
	i.e., $D_g(f) = -D_f(g) = c$.
\end{corollary}
\begin{proof}
	(1)
	Take any polynomial $h \in \mathbb K[x, y]$ such that $f, h$
	are algebraically independent  over $\mathbb K$ and put $g_1 = D_f(h), g_2 = T(h)$.
	Then $g_1 \neq 0$ because in other case $\Ker D_f$ would be of transcendence degree 2
	in $\mathbb K(x, y)$ which is impossible. Note that $(g_2 D_f - g_1 T)(h) = 0$
	and $(g_2D_f - g_1T)(f) = 0$. Since $f, h$ form a transcendence basis of the field $\mathbb K(x, y)$, the next equality holds:
	that $g_2D_f - g_1T = 0$.
	Therefore $T = (g_2 / g_1) D_f$. It follows from the equality
	$[T, D_f] = 0$ that $D_f(g_2 / g_1) = 0$, that is $g_2 / g_1 \in \Ker D_f$. Denoting $\varphi =g_2/g_1$ we get the proof of part (1) of the corollary.
	
	(2)
	Since $T(f) = c$ we have by Lemma \ref{lem1} that $T = D_g$ for some polynomial $g \in \mathbb K[x, y]$.
	But then $T(f) = D_g(f) = \det J(g, f) = c \in \mathbb K^*$.
	The latter means that the polynomials  $f, g$ form a Jacobian pair.
	
\end{proof}

\begin{theorem}\label{Th1}
	Let $f \in \mathbb K[x, y], f = \theta(p)$ for a generative 
	closed polynomial $p \in \mathbb K[x, y]$ with $\deg \theta \geq 1$.
	A derivation $T \in W_2(\mathbb K)$ commutes with $D_f$ if and only if
	$T(p) = \psi(p)$ for some polynomial $\psi(t) \in \mathbb K[t]$ and
	$  \theta''(p)\psi(p) = \theta'(p)( \diver T - \psi'(p)).$
\end{theorem}
\begin{proof}
	Let $[T, D_f] = 0.$ 
	By Lemma \ref{e1}, $\Ker D_f=\mathbb K[p]$, therefore $T(\mathbb K[p])  \subseteq \mathbb K[p].$ Then $T(p)=\psi (p)$ for some polynomial $\psi (t)\in \mathbb K[t].$ 
	Let us prove the equality
	\begin{equation}\label{main} 
		\theta''(p)\psi(p) = \theta'(p)( \diver T - \psi'(p)).
	\end{equation}
	
	First, let $\deg \psi(t) \geq 1$. Write
	$\psi(t) = a_0(t - \lambda_1) \dots (t - \lambda_k)$,
	where $k \geq 1$ and $\lambda_i \in \mathbb K$
	(recall that $\mathbb K$ is algebraically closed).
	Therefore $T(p) = a_0(p - \lambda_1) \dots (p - \lambda_k)$.
	This equality can be written in the form
	$$T(p - \lambda_1) = a_0(p - \lambda_1) \dots (p - \lambda_k)$$
	and taking $p-\lambda _1$ instead of $p$ 
	we can write the last equality as $T(p) = a_0 p (p - \mu_2) \dots (p - \mu_k)$
	for some $\mu_k \in \mathbb K$ (note that the polynomial $p - \lambda_1$
	is also closed and $\mathbb K[p] = \mathbb K[p - \lambda_1]$).
	The last equality can be written in the form
	$$ T(p) = \psi (p)= p \mu(p) \text{ for} \ \mu(p) = a_0 p (p - \mu_2) \dots (p - \mu_k).$$
	By Lemma \ref{lem2}, we have
	\begin{equation}\label{e9}
		[T, D_p] = D_{\psi(p)} - (\diver T) D_p = (\psi'(p) - \diver T)D_p.
	\end{equation}
	But, on the other hand, it follows from the equality 
	$$[T, D_f] = [T, \theta'(p)D_p] = 0$$
	that $T(\theta'(p))D_p = -\theta'(p)[T, D_p]$ and therefore
	\begin{equation}\label{e10}  
		[T, D_p] =- \frac{T(\theta'(p)}{\theta'(p)}D_p
	\end{equation}
	(note that $\theta'(p) \neq 0$ because $f = \theta(p)$ and $f \neq const$).
	
	It follows from (\ref{e9}) and (\ref{e10}) that
	$$ - \frac{T(\theta'(p)}{\theta'(p)} =\psi'(p) - \diver T. $$
	Therefore it holds the desired   equality  $\theta''(p) \psi(p) = \theta'(p)(\diver T - \psi'(p))$ because $T(\theta '(p))=\theta ''(p)\psi (p).$
	
	Now let $\deg \psi(t) < 1$. The latter means that $\psi(t) = c \in \mathbb K$.
	
	By Lemma 3, $T = D_g$ for some polynomial  $g \in \mathbb K[x, y]$ and therefore
	$\diver T = 0$. If $c = 0$, that is $\psi(t) \equiv 0$, then obviously (\ref{main}) holds.
	Let $c \neq 0$. Then $D_g(p) = c$ and the polynomials $p, g$ form a Jacobian pair.
	Let us show that $\deg \theta = 1$ in this case. Indeed, in other case
	$$[T, D_f]=[D_g, D_f] = [D_g, \theta'(p)D_p] = \theta''(p) \cdot c \cdot D_p + \theta '(p)[D_g, D_p].  $$
	By Lemma \ref{lem1}, $[D_g, D_p]=D_c =0$ (recall that $D_g(p)=c$) and therefore $$[T, D_f]=\theta''(p) \cdot c \cdot D_p=0.$$
	The latter contradicts our choice of the derivation $T$ because $c\not =0,$ and $\theta ''\not =0$ by our assumption.
	Therefore  $\deg \theta(t) = 1$.  Taking into account the relations  $\theta''(p) = 0$, 
	$\diver T = 0$, $\psi'(p) = 0$
	we see that the equality (\ref{main}) holds.
	
	Let now 
	$T(p) = \psi(p), f = \theta(p)$ for some closed polynomial $p$, and let the equality (\ref{main}) hold.
	Let us show that $[T, D_f] = 0$, i.e. $[T, \theta'(p)D_p] = 0$.
	The last equality is equivalent to the equality
	\begin{equation}\label{useful}
		T(\theta'(p))D_p = - \theta'(p)[T, D_p]. 
	\end{equation}
	First, consider the case $\deg \psi(t) \geq 1$.
	Then as above one can assume without loss of generality that
	$\psi(t) = t\lambda(t)$ for some polynomial $\lambda(t) \in \mathbb K[t]$.
	By Lemma \ref{lem2}, we get 
	$$[T, D_p] = (\psi'(p) - \diver T) D_p.$$
	Using (\ref{useful}) one can easily show that the latter equality is equivalent to the equality (\ref{main}):
	$$ \theta''(p)\psi(p)D_p = \theta'(p)( \diver T - \psi'(p))D_p, $$
	which holds by our assumptions. So, we have $[T, D_f] = 0$
	in the case $\deg \psi(t) \geq 1$.
	
	Consider the case $\deg \psi(t) < 1$, i.e., $\psi (t)=c$ for some $c\in \mathbb K$ . If $\psi(t) \equiv 0$, then $\diver T=0$ by the equality (\ref{main}).
	Therefore (by Lemma \ref{lem1}) $T = D_g$ for some polynomial $g \in \mathbb K[x, y]$ and  $T(p) = 0 = D_g(p).$ Thus $D_p(g) = 0$, i.e., $g\in \Ker D_p$. 
	It follows from the equality $\Ker D_p = \mathbb K[p]$ that
	$g = \mu(p)$ for some polynomial $\mu(t) \in \mathbb K[t]$.
	But then $T = D_g = \mu'(p)D_p$. Taking into account the equality
	$D_f = D_{\theta(p)} = \theta'(p)D_p$ we get
	$$ [T, D_f] = [\mu'(p)D_p, \theta'(p)D_p] = 0 $$
	because $\mu '(p), \theta '(p) \in \Ker D_p . $
	Let now $\psi(t) = c, c \in \mathbb K^*$.
	Then $T(p) = c$ and from the conditions of the theorem we have
	$\theta''(p) \cdot c = \theta'(p) \diver T$. The latter equality implies that  $\diver T = 0$ because $\diver T $  is a polynomial and $\deg \theta''(p) < \deg \theta'(p) .$
	But then 
	$T = D_g$ for some polynomial $g(t) \in \mathbb K[t]$.
	It follows from the conditions of the theorem that $\theta''(p) = 0$
	and hence $\theta(p) = \alpha p + \beta$ for some
	$\alpha, \beta \in \mathbb K, \alpha \neq 0$.
	Without loss of generality one can assume that $f = \theta(p) = p$.
	We have $T(p) = c$ and $T = D_g$. Then $D_p(g) = -c$ and therefore
	the polynomials $p, g$ form a Jacobian pair. The latter means that
	$$ [T, D_f] = [D_g, D_f] = D_{[p. g]} = D_c = 0$$ 
	that is $T$ and
	$D_f$ commute. The proof of the theorem is complete.
\end{proof}

\begin{corollary}\label{cor4} Let $f \in \mathbb K[x, y]$
	be a closed (in particular, irreducible) polynomial. A derivation
	$T \in W_2(\mathbb K)$ commutes with $D_f$ if and only if
	$T(f) = \psi(f)$ for some polynomial $\psi(t) \in \mathbb K[t]$
	and $\diver T = \psi'(f)$.
\end{corollary}
\begin{proof}
	Since $f$ is closed  we can take without loss of generality thet   $\theta(t) = t.$ Then $\theta''(t) = 0$, and one can easily show that  (\ref{main}) is equivalent to the equality  $\diver T = \psi'(f)$.
\end{proof}
In \cite{Stein}, a class of Jacobian derivations was studied that was induced by weakly semisimple polynomials $f\in \mathbb K[x, y]$ (a polynomial $f$ is called weakly semisimple if the corresponding Jacobian derivation $D_f$ has an eigenfunction $g\in \mathbb K[x, y]$ with nonzero eigenvalue $\lambda \in \mathbb K$, i.e. if $D_f(g)=\lambda g$).  In \cite{Gavr_Step}, such polynomials were described in some cases and some examples were pointed out. Using  some results from \cite{Gavr_Step} one can construct Jacobian derivations whose centralizers are of rank $2$ over their the ring of constants.

\begin{example}
	Let  $f(x, y), g(x, y) \in \mathbb K[x, y]$ be nonzero  polynomials such that $D_f(g)=g.$ Then
	$$  [fD_g-gD_f, D_g]=-D_g(f)D_g-g[D_f, D_g]=gD_g-gD_g=0,  $$
	(here we use the equality $[D_f, D_g]=D_h,$ where $h=D_f(g)$).
	The latter means that the Jacobian derivation $D_g$ has the centralizer in $W_2(\mathbb K)$ of rank $2$ over its ring of constants. This centralizer contains two linearly independent (over $\mathbb K[x, y]$ ) derivations $D_g$ and $fD_g-gD_f.$ Let us choose, for example,  the polynomials  $f$ and $g$ of  the form:  
	$$f(x, y)=x(x-1)y, \ \   g(x, y)=x^3(x-1)y^2.$$ 
	One can easily check that $D_f(g)=g.$  So, the derivation $D_g=-2yx^3(x-1)\partial _x+(4x^3-3x^2)y^2\partial _y$ has the centralizer in $W_2(\mathbb K)$ of rank $2$ over $\mathbb K[p]$ and the corresponding system of differential equations 
	$$ \frac{dx}{dt}=-2yx^3(x-1), \   \frac{dy}{dt}=(4x^3-3x^2)y^2$$
	is integrable (see, for example, \cite{Nagloo}).
\end{example}

\section{On structure of  centralizers of Jacobian derivations }
\begin{theorem}
	Let $f\in \mathbb K[x, y]$ be a nonconstant polynomial and $D_f$ the corresponding Jacobian  derivation. Let $p$ be a   generative closed polynomial   for $f$. Then the centralizer $C_{W_2(\mathbb K)}(D_f)$ is a free module of rank $1$ or $2$ over the subring  $\mathbb K[p]$  of $\mathbb K[x, y].$
\end{theorem}
\begin{proof}
	Since $p$ is a generative closed polynomial for $f$ we have  $f=\theta (p)$  for some polynomial $\theta (t)\in \mathbb K[t].$  Obviously $D_f=\theta '(p)D_p$ and,  by Lemma \ref{e1}, ${\rm ker}D_f={\rm ker} D_p.$
	Let us denote for brevity $C=C_{W_2(\mathbb K)}(D_f)$.   Obviously $C$ is a module over the subring $\mathbb K[p]$ of $\mathbb K[x, y]$. Denote $C_1=\mathbb K[x, y]C$, it is obvious that $C_1$ is a $\mathbb K[x, y]$-module and  ${\rm rk}_{\mathbb K[x, y]}C_1\leq 2$ (recall that $W_2(\mathbb K)$ is a free $\mathbb K[x, y]$-module of rank $2$).
	
	First, let ${\rm rk}_{\mathbb K[x, y]}C_1=1. $ Let us show that in this case ${\rm rk}_{\mathbb K[p]}C=1 $ and $D_p$ is a free generator of the module $C$. Take any $T\in C,  T\not =0.$   By our assumptions on $C_1$, there exist polynomials $g, h\in \mathbb K[x, y]$ such that $gD_p+hT=0$, and at least one of the polynomials $g, h$ is nonzero.  Since $D_p\not =0$ we have $h\not =0$ and therefore $T=-(g/h)D_p.$ But then  $D_f(g/h)=0$. The latter means that the rational function $g/h$ belongs to ${\rm ker }_RD_f,$ where ${\rm ker }_RD_f$ is the kernel of the extension of $D_f$ on $R=\mathbb K(x, y).$ Since ${\rm ker }_RD_f=\mathbb K(p)$ (see, for example \cite{Petien})  we have  $g/h=a(p)/b(p)$ for some polynomials $a(t), b(t)\in \mathbb K[t]$ that can  be chosen to be coprime.  Thus, we obtain the equality $a(p)D_p+b(p)T=0$ and $a(p), b(p)$ are coprime. From the latter  equality  we see that $b(p)$  divides $D_p$, i.e. $D_p=b(p)D_1$ for some derivation $D_1\in W_2(\mathbb K).$ But then $b(p)=const$ because $D_p=-p'_y\partial _x+p'_x\partial _y$ and $b(p)$  does not divide $p'_y, p'_x$ if ${\rm deg}  b(t)>0.$  Thus, we have $T=(-a(p)/b(p))D_p$ and $D_p$ is a free generator for the centralizer $C=C_{W_2(\mathbb K)}(D_f)$ as a $\mathbb K[p]$-module.
	
	Let now ${\rm rk}_{\mathbb K[x, y]}C_1=2. $  Choose any $T\in C$ such that $T$ and $D_p$ are linearly independent over $\mathbb K[x, y]$. It follows from the equality $[T, D_p]=0$ that $T({\rm ker}(D_f))\subseteq  {\rm ker}(D_f)$, so $T(\mathbb K[p])\subseteq \mathbb K[p].$ Therefore $T(p)=\mu (p)$ for some polynomial $\mu (t)\in \mathbb K[t].$  Choose among all such $T\in C$  a derivation $T_0$  such that ${\rm deg}\mu _0(t)$  is minimum, where $\mu _0(t)$ is the corresponding polynomial for $T_0$. One can easily show that for any $T\in C$ its polynomial $\mu (t)$ is divisible by $\mu _0(t)$.  Really, let 
	$$\mu (t)=q(t)\mu _0(t)+r(t)$$ 
	with ${\deg}r(t)<{\rm deg}\mu_0(t)$. Then $T-q(p)T_0\in C$ and $(T-q(p)T_0)(p)=r(p).$ By our choice of $T_0$ we have $r(t)=0$. So, every $T\in C$ can be written in the form $T=q(p)T_0+T_1$, where $T_1=T-q(p)T_0$ satisfies  the equality $T_1(p)=0.$ Since $[T_1, D_p]=0$ we can show using Corollary (\ref{cor1}) that   $T_1=\delta (p)D_p$ for some polynomial $\delta (t)\in \mathbb K[t]$.   Therefore $T=q(p)T_0+\delta (p)D_p.$ The latter means that the derivations $T_0$ and $D_p$ are free generators of the $\mathbb K[p]$-module $C.$  The proof is complete.
\end{proof}


%
\end{document}